\documentclass[reqno,11pt]{amsart}
\usepackage[english]{babel}
\usepackage{amsmath,amsfonts,amsthm,amssymb,amsbsy,upref,color,graphicx,hyperref,enumerate,comment,tikz}
\usepackage[a4paper,margin=2.70truecm]{geometry}
\usepackage{bbm}
\usepackage{soul}

\def\eps{{\varepsilon}}

\def\O{\Omega}

\def\R{\mathbb{R}}

\def\A{\mathcal{A}}

\def\S{\mathcal{S}}

\def\M{\mathcal{M}}

\def\la{\lambda}
\def\eps{\varepsilon}
\def\pa{\partial}

\newcommand{\cp}{\mathrm{cap}}

\newcommand{\be}{\begin{equation}}
\newcommand{\ee}{\end{equation}}
\newcommand{\bib}[4]{\bibitem{#1}{\sc#2: }{\it#3. }{#4.}}

\numberwithin{equation}{section}
\theoremstyle{plain}

\newtheorem{theo}{Theorem}[section]
\newtheorem{lemm}[theo]{Lemma}

\newtheorem{prop}[theo]{Proposition}
\newtheorem{defi}[theo]{Definition}

\theoremstyle{definition}
\newtheorem{rema}[theo]{Remark}

\title[]{On a reverse Kohler-Jobin inequality}

\author[L. Briani]{Luca Briani}
\author[G. Buttazzo]{Giuseppe Buttazzo}
\author[S. Guarino Lo Bianco]{Serena Guarino Lo Bianco}

\date{}

\begin{document}

\begin{abstract}
We consider the shape optimization problems for the quantities $\lambda(\O)T^q(\O)$, where $\O$ varies among open sets of $\R^d$ with a prescribed Lebesgue measure. While the characterization of the infimum is completely clear, the same does not happen for the maximization in the case $q>1$. We prove that for $q$ large enough a maximizing domain exists among quasi-open sets and that the ball is optimal among {\it nearly spherical domains}.
\end{abstract}

\maketitle



\section{Introduction}\label{sintro}

In the present paper we consider two well-known quantities that occur in the study of elliptic equations in the Euclidean space $\R^d$, $d\ge2$. The first one is usually called {\it torsional rigidity} and is defined, for every nonempty open set $\O\subset\R^d$ with finite Lebesgue measure (in the following a {\it domain}), as
$$T(\O)=\int w_\O\,dx,$$
where $w_\O$ is the unique solution of the PDE
$$-\Delta u=1\quad\hbox{in }\O,\qquad u\in H^1_0(\O).$$
Equivalently, we may define $T(\O)$ as
$$T(\O)=\max\left\{\Big[\int u\,dx\Big]^2\Big[\int|\nabla u|^2\,dx\Big]^{-1}\ :\ u\in H^1_0(\O)\setminus\{0\}\right\}.$$
In the integrals above and in the following we use the convention that integrals without the indicated domain are intended over the entire space $\R^d$. The quantity $T(\O)$ verifies the scaling property
$$T(t\O)=t^{d+2}T(\O)\qquad\text{for every }t>0;$$
in addition, the maximum of $T(\O)$ among domains with prescribed measure is reached by the ball ({\it Saint Venant inequality}), which can be written in the scaling free formulation as
$$|\O|^{-(d+2)/d}T(\O)\ge|B|^{-(d+2)/d}T(B),$$
for every domain $\O$ and for every ball $B\subset \R^d$. 

The second quantity is the {\it first eigenvalue} $\lambda(\O)$ of the Dirichlet Laplacian, defined as the smallest $\lambda$ such that the PDE
$$-\Delta u=\lambda u\quad\hbox{in }\O,\qquad u\in H^1_0(\O)$$
admits a nonzero solution. Equivalently, $\lambda(\O)$ can be defined through the minimization of the Rayleigh quotient
$$\lambda(\O)=\min\left\{\Big[\int|\nabla u|^2\,dx\Big]\Big[\int u\,dx\Big]^{-2}\ :\ u\in H^1_0(\O)\setminus\{0\}\right\}.$$
The quantity $\lambda(\O)$ verifies the scaling property
$$\lambda(t\O)=t^{-2}\lambda(\O)\qquad\text{for every }t>0;$$
in addition, the minimum of $\lambda(\O)$ among domains with prescribed measure is reached by the ball ({\it Faber-Krahn inequality}), which can be written in the scaling free formulation as
$$|\O|^{2/d}\lambda(\O)\ge|B|^{2/d}\lambda(B),$$
for every domain $\O$ and for every ball $B\subset \R^d$.

The study of relations between $T(\O)$ and $\lambda(\O)$ was performed in several papers (see for instance \cite{BBP21}, \cite{BBV}, \cite{bfnt16}, \cite{bra14}, \cite{BBP22}, \cite{bgm17}, \cite{BP21}, \cite{dpga14}, \cite{kj78a}, \cite{kj78b}, \cite{luzu19}), where some important inequalities were established. In particular:
\begin{itemize}
\item[-]the Kohler-Jobin inequality
$$\lambda(\O)T^q(\O)\ge\lambda(B)T^q(B),$$
valid for every $q\in[0,2/(d+2)]$ and for every domain $\O$, where $B$ is any ball in $\R^d$ with $|B|=|\O|$;

\item[-]the P\'olya inequality
$$0<\frac{\lambda(\O)T(\O)}{|\O|}<1,$$
valid for every domain $\O$ of $\R^d$.
\end{itemize}

In the present paper we consider the scaling free shape functional
$$F_q(\O)=\frac{\lambda(\O)T^q(\O)}{|\O|^{\alpha_q}}, \quad\quad \text{ with } \alpha_q=\frac{-2+q(d+2)}{d},$$
and the two quantities
$$\begin{cases}
m_q=\inf\big\{F_q(\O)\ :\ \O\text{ domain}\big\};\\
M_q=\sup\big\{F_q(\O)\ :\ \O\text{ domain}\big\}.
\end{cases}$$
While the situation for $m_q$ is fully clear, and by Kohler-Jobin inequality, together with the Saint Venant inequality, we have
$$m_q=\begin{cases}
F_q(B)&\text{if }q\le2/(d+2)\\
0&\text{if }q>2/(d+2),
\end{cases}$$
the characterization of $M_q$ is not yet complete. The results available up to now are (see \cite{BBP21} and \cite{bfnt16}):
\begin{itemize}
\item[]$M_q=\infty$ for every $q<1$;
\item[]$M_q=1$ when $q=1$, with the upper bound $1$ not reached by any domain $\O$;
\item[]$M_q<\infty$ for every $q>1$.
\end{itemize}
We investigate here this last case. The maximal expectation would be having the following result (reverse Kohler-Jobin inequality):
\begin{itemize}
\item[-]for every $q>1$ the supremum $M_q$ is reached on an optimal domain $\O_q$;
\item[-]there exists a threshold $q^*>1$ such that for every $q\ge q^*$ the supremum $M_q$ is reached by a ball.
\end{itemize}
We are unable to prove the results in the strong form above, and we prove here the weaker results below:
\begin{itemize}
\item[-]for every $q>1$ the supremum $M_q$ is reached on a capacitary measure $\mu_q$ (Theorem \ref{existence});
\item[-]there exists a threshold $q_0>1$ such that for every $q\ge q_0$ the supremum $M_q$ is reached by a domain $\O_q$ (Theorem \ref{extDomain});
\item[-]there exists another threshold $q_1$ such that for every $q\ge q_1$ the ball is a maximizer for the shape functional $F_q$ among {\it nearly spherical domains} (Theorem \ref{maxBall}).
\end{itemize}

While finishing this paper we have been informed that similar problems are considered in the work in progress \cite{BLNP}.

\section{Capacitary measures}\label{scap}

The concept of capacitary measure and the related properties is a very useful tool for our purposes. When dealing with sequences of PDEs of the form
$$-\Delta u=f\quad\hbox{in }\O_n,\qquad u\in H^1_0(\O_n),$$
a natural question is to establish if the sequence $u_{n,f}$ of solutions, or a subsequence of it, converges in $L^2$ to some function $u_f$ and to determine in this case the PDE that the function $u_f$ solves. Starting from the pioneering papers \cite{dm88}, \cite{dmmo87} is now well understood that the right framework to treat such a kind of questions is that of capacitary measures. Below we recall the main results and definitions following \cite{BBV14} and \cite{V15}. For further information we refer the reader to the monographs \cite{BB05}, \cite{HP18} and references therein.

\begin{defi}
We say that a nonnegative Borel regular measure $\mu$, possibly taking the value $\infty$, is a capacitary measure if
$$\mu(E)=0\text{ whenever $E$ is a Borel set with }\cp(E)=0,$$
being $\cp(E)$ the capacity
$$\cp(E)=\inf\Big\{\int_{\R^d}|\nabla u|^2+u^2\,dx\ :\ u\in H^1_0(\R^d),\ u=1\text{ in a neighborhood of }E\Big\}.$$
\end{defi}

A property $P(x)$ is said to hold quasi-everywhere (briefly q.e.) if the set where $P(x)$ does not hold has zero capacity. 
A Borel set $\O\subset\R^d$ is said to be quasi-open if there exists a function $u\in H^1(\R^d)$ such that $\O=\{u>0\}$ up to a set of capacity zero. A function $f:\R^d\to \R$ is said to be quasi-continuous if there is a sequence of open sets $\omega_n\subset\R^d$ such that $\lim_{n\to\infty}\cp(\omega_n)=0$ and $f$ is continuous when restricted to $\R^d\setminus \omega_n$. It is well known (see for instance \cite{EvGa}) that every Sobolev function has a quasi-continuous representative, and that two quasi-continuous representatives coincide quasi-everywhere. We then identify the space $H^{1}(\R^d)$ with the space of quasi-continuous representatives. We recall that a sequence $u_n\in H^{1}(\R^d)$ that converges in norm to some $u\in H^{1}(\R^d)$, converges quasi-everywhere (up to a subsequence) to $u$. 

Given $\mu$ a capacitary measure we denote by $H^1_\mu$ the following space
\[
H^1_{\mu}=H^1(\R^d)\cap L^{2}_\mu(\R^d) = \left\{u\in H^1(\R^d)\, :\, \int u^2\,d\mu<\infty\right\}.
\]
The space $H^1_\mu$ is an Hilbert space when endowed with 
$\|u\|_{H^{1}_\mu}=\|u\|_{H^1(\R^d)}+\|u\|_{L^2_\mu(\R^d)}$, where the quantity $\|u\|_{L^2_\mu(\R^d)}$ is well defined, being Sobolev functions defined up to a set of zero capacity.
We always identify two capacitary measures $\mu,\nu$ for which
\be\label{identifymeasure}
\int u^2d\mu=\int u^2d\nu, \hbox{ for every } u\in H^{1}(\R^d).
\ee
If instead \eqref{identifymeasure} holds with ``$\le$'' we say that $\mu\le\nu$, and in this case we have $H^1_\nu\subseteq H^1_\mu$. We can associate to any open set (or more generally to any quasi-open set) $\O\subset\R^d$ the capacitary measure $I_\O$ defined as follows
$$I_\O(E):=
\begin{cases}
0&\hbox{ if }\cp(E\setminus\O)=0,\\
\infty&\hbox{ if }\cp(E\setminus\O)>0.
\end{cases}$$
Notice that, if $\mu=I_\O$ for some open set $\O\subset\R^d$, then $H^{1}_{\mu}=H^{1}_0(\O)$. 

To extend the notion of torsional rigidity to a capacitary measure $\mu$ we need to carefully deal with the fact that the embedding $H^1_\mu\hookrightarrow L^1(\R^d)$ can be noncompact and even noncontinuous. Nevertheless we can follow an approximation argument: for every $R>0$, let $w_R$ be the solution to the following minimization problem
$$
\min\left\{\int|\nabla u|^2\,dx+\int u^2\,d\mu-\int u\,dx\ :\ u\in H^{1}_\mu\cap H^{1}_0(B_R)\right\}
$$
The torsion function $w_\mu$ and the torsional rigidity $T(\mu)$ of the capacitary measure $\mu$ are defined as:
$$
w_{\mu}:=\sup_{R>0}w_R, \quad T(\mu):=\int w_{\mu}dx.
$$
The Dirichlet eigenvalue of $\mu$ can be defined through the following Rayleigh-type quotient:
\[
\lambda_1(\mu)= \inf_{u\subset H^1_\mu\setminus\{ 0\}} \frac{\int|\nabla u|^2\,dx+\int u^2\, d\mu}{\int u^2\, dx}.
\]
Clearly, if $\mu=I_\O$ for some domain $\O\subset\R^d$, we have $T(\mu)=T(\O)$ and $\la(\mu)=\la(\O)$ (we adopt this notation also if $\O$ is a quasi-open set). For a general capacitary measure $\mu$, neither $\la(\mu)$ is necessarily attained by some function $u\in H^{1}_\mu$ nor $T(\mu)$ is necessarily finite. However, as shown in \cite{BB12}, it holds the following:
$$
w_\mu\in L^{1}(\R^d)\Longleftrightarrow T(\mu)<\infty\Longrightarrow \la_1(\mu) \hbox{ is attained by some }u \in H^{1}_\mu.
$$
For every capacitary measure $\mu$ with $T(\mu)<\infty$ we define the {\it set of finiteness} $A_\mu$ as the quasi-open set
$$A_\mu:=\{w_{\mu}>0\}.$$
In the case when $\mu=I_\O$, for some domain $\O\subset\R^d$, we have $A_\mu=\O$. The set of capacitary measures with finite torsion can be endowed with the following notion of distance. 

\begin{defi}
Given two capacitary measures $\mu,\nu$ such that $w_\mu,w_\nu\in L^{1}(\R^d)$ we define the $\gamma-$distance between them as $d_\gamma(\mu,\nu)=\|w_\mu-w_\nu\|_{L^1(\R^d)}$. We say that a sequence $\mu_n$ $\gamma-$converges to $\mu$ if $d_\gamma(\mu_n,\mu)\to 0$ as $n\to\infty$. When $I_{\O_n}\overset{\gamma}\to\mu$ we simply write $\O_n\overset{\gamma}\to\mu$.
\end{defi}

We summarize the main properties of the $\gamma-$distance below:
\begin{itemize}
\item The space $(\{\mu:\ \mu \hbox{ capacitary measure with }w_\mu\in L^{1}(\R^d)\},d_{\gamma})$ is a complete metric space and the set $\{I_\O:\ \O\subset\R^d \hbox{ open set with }w_\O\in L^{1}(\R^d)\}$ is a dense subset of it.
\item The functionals $\mu\mapsto\lambda(\mu)$ and $\mu\mapsto T(\mu)$ are $\gamma-$continuous.
\item The map $\mu\mapsto|A_\mu|$, or more generally integral functionals as $\int_{A_\mu}f(x)\,dx$ with $f\ge0$ and measurable, are lower semicontinuous with respect to the $\gamma$-convergence.
\item The $\gamma$-convergence of $\mu_n$ to $\mu$ implies the $\Gamma$-convergence in $L^2(\R^d)$ of the functionals $\|\cdot\|_{H^{1}_{\mu_n}}:L^2(\R^d)\to L^{2}(\R^d)$ defined by
$$\|u\|_{H^{1}_\mu}=
\begin{cases}
\|u\|_{H^{1}(\R^d)}+\int u^2\,d\mu_n\ & \hbox{if }u\in H^{1}_{\mu_n}\\ 
\infty & \hbox{if } u\not\in H^{1}
\end{cases}
$$
to the functional $\|\cdot\|_{H^{1}_\mu}:L^2(\R^d)\to L^{2}(\R^d)$ ,
$$\|u\|_{H^{1}_\mu}=
\begin{cases}
\|u\|_{H^{1}(\R^d)}+\int u^2\,d\mu\ & \hbox{if }u\in H^{1}_\mu
\\ \infty & \hbox{if } u\not\in H^{1}.
\end{cases}
$$
\item For a given capacitary measures $\mu$ with finite torsion we call resolvent of $\mu$ the linear compact and self-adjoint operator
$$R_\mu:L^{2}(\R^d)\to L^{2}(\R^d),\quad R_\mu(f)=w_{\mu,f},$$
where $w_{\mu,f}$ is the solution of the problem 
$$
w_{\mu,f}\in H^{1}_\mu,\quad -\Delta w_{\mu,f}+w_{\mu,f}\mu =f, 
$$
in the sense that
$$
w_{\mu,f}\in H^1_\mu,\quad \int \nabla w_{\mu,f}\cdot \nabla \phi dx+\int w_{\mu,f}\phi d\mu=\int f\phi dx \ \hbox{ for every }\phi\in H^{1}_\mu.
$$
The $\gamma$-convergence of $\mu_n$ to $\mu$ implies the norm convergence of $R_{\mu_n}$ to $R_{\mu}$, i.e.
$$
\lim_{n\to\infty}\|R_{\mu_n}-R_{\mu}\|_{\mathcal{L}(L^{2}(\R^d),L^{2}(\R^d))}=0.
$$
\item If $\mu_n$ is a sequence of capacitary measures whose set of finiteness have uniformly bounded measures $|A_{\mu_n}|$, then
$$\mu_n\overset{\gamma}{\to}\mu\Longleftrightarrow \|R_{\mu_n}-R_{\mu}\|_{\mathcal{L}(L^{2}(\R^d),L^2(\R^d))}\to 0\Longleftrightarrow \|u\|_{H^{1}_{\mu_n}}\overset{\Gamma}
{\longrightarrow}\|u\|_{H^{1}_{\mu}}\hbox{ on }L^2(\R^d).$$
\end{itemize}

The classical concentration-compactness principle of P.L. Lions was extended to sequences of open sets in \cite{B00}. Notably, the following result holds.

\begin{theo}\label{Bucurconcomp}
Let $\O_n$ be a sequence of open sets with uniformly bounded measures. Then there exists a subsequence (still denoted with the same indices $n$) such that one of the following situations occurs.
\begin{itemize}
\item[-]Compactness: there exists a sequence $x_n\subset\R^d$ such that the sequence of capacitary measures $I_{\O_n}(x_n+\cdot)$ $\gamma-$converges. 
\item[-]Vanishing: the sequence $R_{I_{\O_n}}$ converges in norm to $0$. Moreover we have $\|w_{\O_n}\|_{L^{\infty}}\to 0$ and $\lambda(\O_n)\to\infty$, as $n\to\infty$.
\item[-]Dichotomy: there exist two sequences of quasi-open sets $\O^1_{n},\O_n^2\subset\O_n$ such that
\begin{itemize}
\item[-]$\mathrm{dist}(\O^1_n,\O^2_{n})\to \infty$, as $n\to\infty$;
\item[-]$d_{\gamma}(I_{\O_m},I_{\O^1_n\cup \O^2_n})\to 0$, as $n\to\infty$;
\item[-]$\liminf_{n\to\infty}T(\O^1_n)>0$ and $\liminf_{n\to\infty} T(\O^2_n)>0$.
\end{itemize}
\end{itemize}
\end{theo}

The proof of the theorem above can be deduced by combining Theorem 2.2 of \cite{B00} and Theorem 3.5 of \cite{BBV14}.

\section{Relaxation of $F_q$}\label{srel}

In this section we characterize the relaxation of the functional $F_q$ to the set of capacitary measures. We define the set $\M_{ad}$ of admissible capacitary measures as
$$
\M_{ad}=\{\mu\ :\ \mu \hbox{ capacitary measure with }\ 0<|A_\mu|<\infty\}.
$$
For $\mu\in\M_{ad}$ we define the relaxed form of our functional $F_q$ as
$$F_q(\mu)=\sup\Big\{\limsup_n F_q(\O_n)\ :\ \O_n\subset\R^d \hbox{ open set such that } \O_n\overset{\gamma}{\to}\mu\Big\},$$
so that 
\[M_q=\sup\{F_q(\mu) \,: \, \mu\in\M_{ad} \}.\]

\begin{lemm}\label{amu}
Let $\mu\in \M_{ad}$ and $\O_n$ a sequence of domains such that $\O_n\overset{\gamma}{\to}\mu$. If $|A_\mu|< \infty$ then $\O_n\cap A_\mu\overset{\gamma}{\to}\mu$.
\end{lemm}

\begin{proof}
Being the sequence $\O_n\cap A_\mu$ of uniformly bounded measure, by the properties of $\gamma$-convergence seen above we have to show that
$$\|u\|_{H^{1}_{\mu_n}}\overset{\Gamma}
{\longrightarrow}\|u\|_{H^{1}_{\mu}}\hbox{ on }L^2(\R^d),$$
where we set $\mu_n=I_{\O_n\cap A_\mu}$.

The ``$\Gamma$-liminf'' inequality readily follows by the fact that $H^{1}_{\mu_n}=H^{1}_0(\O_n\cap A_{\mu})\subseteq H^{1}_0(\O_n)$ and by the $\Gamma$ convergence of $\|\cdot\|_{H^{1}_0(\O_n)}$ to $\|\cdot\|_{H^{1}_\mu}$ in $L^{2}(\R^d)$. 

To prove the ``$\Gamma$-limsup'' inequality we can suppose without loss of generality that $u\in H^{1}_{\mu}$.
Since $\O_n\overset{\gamma}{\to}\mu$, there exists a sequence $u_n\in H^1_0(\O_n)$ such that
\[
\begin{split}
&u_n\longrightarrow u\hbox{ strongly }L^2(\R^d),\\
&\lim_{n\to\infty}\left(\int |\nabla u_n|^2dx\right)= \int |\nabla u|^2dx+\int |u|^2 d\mu.
\end{split}
\]
We denote respectively by $u_n^+$ and $u_n^-$ the positive and negative part of $u_n$. 
Since we have
$$
\int |\nabla (u_n^+-u_n^-)|^2dx=\int|\nabla u_n^+|^2dx+\int|\nabla u_n^-|^2dx,
$$
and $u_n=u_n^+-u_n^-$, by possibly passing to a subsequence (still indexed by $n$) we can suppose that
\be\label{utile}
\begin{split}
&\limsup_{n\to\infty}\left(\int|\nabla u_n^+|^2dx\right)+\limsup_{n\to\infty}\left(\int|\nabla u_n^-|^2dx\right)=\lim_{n\to\infty}\left(\int |\nabla (u_n^+-u_n^-)|^2dx\right)\\
&=\int |\nabla u|^2dx+\int u^2 d\mu.
\end{split}
\ee
We define 
$$v_n^+=u_n^+\wedge u^+\in H^1(\R^d),\qquad v_n^-=u_n^-\wedge u^-\in H^{1}(\R^d).$$
Since $u\in H^{1}_\mu$ and $u_n\in H^{1}_0(\O_n)$ we have $u=0$ q.e. on $A_\mu^c$ and $u_n=0$ q.e. on $\O_n^c$.
This implies that both $v_n^+$ and $v_n^-$ vanish q.e. on $(\O_n\cap A_\mu)^c$ and consequently that $v_n^+, v_n^-\in H^{1}_0(\O_n\cap A_\mu)$. Moreover it is easy to show that
$$v_n^+-v_n^-\longrightarrow u, \quad \hbox{strongly } L^{2}(\R^d).$$ 
Therefore the thesis is achieved if we show that
\be\label{goal}
\limsup_{n\to\infty}\left(\int |\nabla (v_n^+-v_n^-)|^2\,dx\right)\le \lim_{n\to\infty}\left(\int |\nabla (u_n^+-u_n^-)|^2\,dx\right).
\ee
We have
\be\label{prima}
\begin{split}
\int|\nabla v^+_n|^2\,dx&=\int_{\{u^+_n\le u^+\}}|\nabla u^+_n|^2\,dx+\int_{\{u^+_n>u^+\}}|\nabla u^+|^2\,dx\\
&=\int|\nabla u^+_n|^2\,dx-\int\Big(|\nabla u^+_n|^2-|\nabla u|^2\Big)1_{\{u^+_n>u^+\}}\,dx.
\end{split}
\ee
By lower semicontinuity we have
\be\label{seconda}
\liminf_n\int\Big(|\nabla u^+_n|^2-|\nabla u^+|^2\Big)1_{\{u^+_n>u^+\}}\,dx\ge0.
\ee
Indeed, to show the inequality above, it is enough to write
\[\begin{split}
\int\Big(|\nabla u^+_n|^2-|\nabla u^+|^2\Big)1_{\{u^+_n>u^+\}}\,dx
&=\int\Big(|\nabla u^+_n|^2-|\nabla u^+|^2\Big)1_{\{u^+_n\ge u^+\}}\,dx\\
&=\int|\nabla(u^+_n\vee u^+)|^2-|\nabla u^+|^2\,dx
\end{split}\]
and to notice that $u^+_n\rightharpoonup u^+$ weakly in $H^1(\R^d)$ implies $u^+_n\vee u^+\rightharpoonup u^+$ weakly in $H^1(\R^d)$ and so, by lower semicontinuity
$$\liminf_n\int|\nabla(u^+_n\vee u^+)|^2-|\nabla u^+|^2\,dx\ge0.$$
Combining \eqref{prima} and \eqref{seconda} we deduce that
\be\label{terza}
\limsup_{n\to\infty}\left(\int|\nabla v^+_n|^2\right)\le \limsup_{n\to\infty}\left(\int|\nabla u^+_n|^2\,dx\right).
\ee
Similarly we have
\be\label{quarta}
\limsup_{n\to\infty}\left(\int|\nabla v^-_n|^2\right)\le \limsup_{n\to\infty}\left(\int|\nabla u^-_n|^2\,dx\right).
\ee
Combining \eqref{utile}, \eqref{terza} and \eqref{quarta} we finally deduce \eqref{goal} and this concludes the lemma.
\end{proof}

\begin{rema}\label{rema}
By Lemma \ref{amu} for every measure $\mu\in\M_{ad}$ there exists a sequence of quasi-open sets $\O_n$ (that can be taken open by a standard approximation procedure) such that $I_{\O_n}$ $\gamma-$converges to $\mu$ and for which
$$|\O_n|\to |A_{\mu}|\quad \hbox{ as $n\to\infty$}.$$ 
This in turns implies that the set 
$$\{I_\O:\ \O\subset\R^d\hbox{ domain}\}$$
is $\gamma-$dense in $\M_{ad}$. Furthermore, we can extend both Saint-Venant, Faber-Krahn and P\'olya inequalities to any capacitary measure. That is
\be\label{FKSV}
|A_\mu|^{-(d+2)/d}T(\mu)\le|B|^{-(d+2)/d}T(B),\qquad |A_\mu|^{2/d}\lambda(\mu)\ge|B|^{2/d}\lambda(B),
\ee
and
\be\label{P}
0<|A_{\mu}|^{-1}\lambda(\mu)T(\mu)<1
\ee
for every measure $\mu\in\M_{ad}$ and every ball $B\subset\R^d$.
\end{rema}

\begin{prop}\label{relamu}
Let $\mu\in\M_{ad}$. Then we have
\be\label{amu1}
|A_\mu|=\inf\Big\{\liminf_n|\O_n|\ :\ \O_n\text{ domain, }\O_n\overset{\gamma}{\to}\mu\Big\}.
\ee
The quantity $|A_\mu|$ is then the relaxation, in the $\gamma$-convergence, of the Lebesgue measure $|\O|$. As a consequence, we have
\be\label{fqmu}
F_q(\mu)=\frac{\lambda(\mu)T^q(\mu)}{|A_\mu|^{\alpha_q}}.
\ee
\end{prop}

\begin{proof}
The inequality $\le$ in \eqref{amu1} follows from the $\gamma$-lower semicontinuity of the map $\mu\mapsto|A_\mu|$ seen above. The opposite inequality follows at once by Remark \ref{rema}. Since $T(\mu)$ and $\lambda(\mu)$ are $\gamma$-continuous, the proof of \eqref{fqmu} is achieved by a similar argument.
\end{proof}

The scaling properties of the shape functionals $|\O|$, $\lambda(\O)$, $T(\O)$ and $F_q(\O)$ extend to their relaxations $|A_\mu|$, $\lambda(\mu)$, $T(\mu)$ and $F_q(\mu)$ in $\M_{ad}$. More precisely, setting for $t>0$
\[\mu_t(E)=t^{d-2}\mu(E/t),\] 
we have
\[
|A_{\mu_t}|=t^d|A_\mu|, \quad \lambda(\mu_t)=t^{-2}\lambda(\mu), \quad T(\mu_t)=t^{d+2}T(\mu), \quad F_q(\mu_t)=F_q(\mu).
\]

\section{Existence of an optimal measure for $q>1$}\label{sexi}

In \cite{bfnt16} it is proved that the supremum $M_1=1$ is not attained in the class of domains. In the next proposition we point out that the same occurs even in the class $\M_{ad}$.

\begin{prop}[Nonexistence for $q=1$ of an optimal measure]\label{q1nomeas}
Given a capacitary measure $\mu\in \M_{ad}$ the problem
$\sup\{F_1(\mu):\ \mu\in\M_{ad}\}$ does not have a maximizer.
\end{prop}

\begin{proof}
The proof follows at once by exploiting Theorem 1.1. of \cite{bfnt16} which asserts that there exists a dimensional constant $c_d>0$ for which
\be\label{polyafine}
F_1(\O)\le 1-\frac{c_d T(\O)}{|\O|^{1+\frac{2}{d}}},
\ee
for every domain $\O$. Then, for every $\mu\in\M_{ad}$, by Remark \ref{rema} we can select a sequence $\O_n\overset{\gamma}{\to}A_\mu$ for which 
$$
F_1(\O_n)\to F(\mu), \quad T(\O_n)\to T(\mu), \quad |\O_n|\le |A_\mu| \quad \text{as } n\to\infty .
$$
Thus, using \eqref{polyafine} with $\O=\O_n$ and passing to the limit as $n\to\infty$, we get $F_1(\mu)<1=M_1$.
\end{proof}

To prove the main result of this section we need the following elementary lemma.

\begin{lemm}\label{elementary}
Let $0<c_1<c_2<\infty$, $1<\alpha_1<\alpha_2<\infty$. Then, there exists $\beta<1$ such that, for every $a,b,c,d\in (c_1,c_2)$ it holds
$$
\frac{(a+b)^{\alpha_1}}{(c+d)^{\alpha_2}}\le \beta\max\left\{ \frac{a^{\alpha_1}}{c^{\alpha_2}}, \frac{b^{\alpha_1}}{d^{\alpha_2}}\right\}.
$$
\end{lemm}

\begin{proof}
Letting $x=b/a$ and $y=d/c$, is enough to prove that
$$
\frac{(1+x)^{\alpha_1}}{(1+y)^{\alpha_2}}\le \beta \max\left\{1, \frac{x^{\alpha_1}}{y^{\alpha_2}}\right\}.
$$
Suppose that $x\le y$. Since $x\ge \frac{c_1}{c_2}$, it holds
\be\label{caso1}
(1+x)^{\alpha_1}= (1+x)^{\alpha_2}(1+x)^{\alpha_1-\alpha_2}\le (1+y)^{\alpha_2}\left(1+\frac{c_1}{c_2}\right)^{\alpha_1-\alpha_2}.
\ee
Similarly, if $x>y$, since $x\le \frac{c_2}{c_1}$, it holds
\be\label{caso2}
\left(1+\frac{1}{x}\right)^{\alpha_1}\le \left(1+\frac{1}{y}\right)^{\alpha_2}\left(1+\frac{1}{x}\right)^{\alpha_1-\alpha_2}\le \left(1+\frac{1}{y}\right)^{\alpha_2}\left(1+\frac{c_2}{c_1}\right)^{\alpha_1-\alpha_2}.
\ee
Eventually we achieve the thesis by letting 
$$\beta=\left(1+\frac{c_1}{c_2}\right)^{\alpha_1-\alpha_2}$$
and combining \eqref{caso1} and \eqref{caso2}.
\end{proof}

\begin{theo}[Existence for $q>1$ of an optimal measure] \label{existence}
For every $q>1$ there exists a measure $\mu^\star\in \M_{ad}$ such that
$$F_q(\mu^\star)=\sup \left\{ F_q(\mu)\, : \, \mu\in \M_{ad} \right\}.$$
\end{theo}

\begin{proof}
We select a sequence $\mu_n\in \M_{ad}$ such that $
F_q(\mu_n)\to M_q$, as $n\to\infty$. By density, we can suppose that $\mu_n=I_{\O_n}$, for some sequence of open sets $\O_n$. Further, being $F_q$ scaling free, we can also assume $|\O_n|=1$.
Hence, we can apply Theorem \ref{Bucurconcomp}.

If dichotomy occurs, then there exist two sequences of quasi-open sets $\O_n^1,\O_n^2\subset\O_n$ such that 
$$\O_n^1\cap \O_n^2=\emptyset,\quad d_{\gamma}(I_{\O_n}, I_{\O^1_n\cup\O^2_n})\to0 \quad \text{ as }n\to \infty.$$
Taking into account the Saint-Venant inequality and the fact that $|\O_n|=1$, there exist constants $c_1,c_2>0$, which depend only on the dimension, such that 
$$
c_1<\inf_{n}|T(\O_n^i)|\le\sup_{n}|T(\O_n^i)|<c_2, \quad c_1<\inf_{n}|\O_n^i|\le\inf_{n}|\O_n^i|<c_2, \hbox{ for $i=1,2$}.
$$
Since $\lambda_1$ is increasing with respect to set inclusion, we have
\be\label{lambdautile}
\lambda_1(\O_n)\le\min\{\lambda_1(\O_n^1),\lambda(\O^2_n)\}.
\ee
Lemma \ref{elementary} together with \eqref{lambdautile} gives
$$\frac{\la(\O_n)\left(T(\O_n^1\cup \O_n^2)\right)^q}{|\O_n|^{\alpha_q}}\le \frac{\la(\O_n)\left(T(\O_n^1)+T(\O_n^2)\right)^q}{(|\O_n^1|+|\O_n^2|)^{\alpha_q}}\le\beta \max_{i=1,2}\frac{\la(\O_n^i)T^q(\O_n^i)}{|\O_n^i|^{\alpha_q}}<F_q(\O_n).$$
By taking the limit for $n\to\infty$ in the latter inequality we obtain the contradiction
$$\sup_{\mu\in\M_{ad}}F(\mu)<\sup_{\mu\in\M_{ad}}F(\mu),$$
and hence dichotomy cannot occur.
Now, the maximality condition on the sequence $\O_n$ together with P\'olya inequality gives that for $n$ large enough
\be\label{ineut}
\la(B)T^q(B)/|B|^{\alpha_q}\le\lambda(\O_n)T^q(\O_n)=\lambda(\O_n)T(\O_n)\cdot T^{q-1}(\O_n)\le T^{q-1}(\O_n),
\ee
where $B$ is any ball of $\R^d$. In particular it cannot be $\lim_{n\to\infty}T(\O_n)=0$, and this rules out the vanishing case.

Therefore compactness holds and there exists a capacitary measure $\mu^\star$ and a sequence $x_n\in\R^d$ such that $I_{x_n+\O_n}$ $\gamma-$converges to $\mu^\star$.

By \eqref{ineut} we deduce that $T(\mu^\star)>0$ which by \eqref{FKSV} implies $|A_{\mu^\star}|>0$ and hence that $\mu^\star$ belongs to $\M_{ad}$. Clearly the measure $\mu^\star$ maximizes the functional $F_q$ on $\M_{ad}$ and this concludes the proof.
\end{proof}

\section{Optimal measures are quasi-open sets for large $q$}\label{sopt}

We are now interested to prove that, when $q$ is large enough, optimal measures $\mu$ coming from Theorem \ref{existence} can be represented as quasi-open sets. We begin by recalling the following result, see \cite{Da} and \cite{V15} Proposition 3.83.

\begin{theo}\label{Davies}
Let $\mu$ be a capacitary measure with finite torsion. Then the eigenfunctions $u\in L^2(\R^d)$ of the operator $-\Delta +\mu$ with unitary $L^2$ norm are in $L^{\infty}(\R^d)$ and satisfy 
$$\|u\|_\infty\le e^{1/(8\pi)}\lambda(\mu)^{d/4}.$$
\end{theo}

We also use the following lemma.


\begin{lemm}\label{qtoinfty}
For every $q>1$ let $\mu_q\in\M_{ad}$ be a maximal measure for the functional $F_q$, such that $|A_{\mu_q}|=1$. Then 
$$\liminf_{q\to\infty}T(\mu_q)>0.$$
\end{lemm}

\begin{proof}
Let $q_n$ be a diverging sequence and $B\subset \R^d$ be a ball of unitary measure. By a standard diagonal argument we can select a sequence $\O_n\subset\R^d$ of open sets such that $|\O_n|= 1$ for every $n$ and
\be\label{opiccolo}
|F_{q_n}(\O_n)-F_{q_n}(\mu_{q_n})|=o(T^{q_n}(B)) \quad \hbox{as $n\to\infty$}.
\ee
Then we can apply Theorem \ref{Bucurconcomp} to the sequence $\O_n$. 
Dichotomy can be ruled out by the same argument as in the proof of Theorem \ref{existence} once noticed that a combination of \eqref{FKSV} and \eqref{P} implies
$$F^{1/{q_n}}_{q_n}(\mu)\leq T^{(q_n-1)/q_n}(B)\rightarrow T(B) \, \text{ as } n\rightarrow \infty.$$
The vanishing case can be excluded too by following again the proof of Theorem \ref{existence}. Indeed, for $n$ large enough, P\'olya inequality and \eqref{opiccolo} imply
$$
T^{q_n-1}(\O_n)\ge F_{q_n}(\O_n)\ge F_{q_n}(\mu_{q_n})-|F_{q_n}(\O_n)-F_{q_n}(\mu_{q_n})| \ge F_{q_n}(B)+o(T^{q_n}(B)). 
$$
Hence we deduce
$$\liminf_{n\to\infty} T^{(1-1/q_n)}(\O_n)>0,$$
which implies that it cannot be $T(\O_n)\to 0$, as $n\to\infty$. Therefore compactness holds true and the sequence $\O_n$ has a subsequence (still denoted by the same indices) that $\gamma$-converges to some $\mu \in\M_{ad}$ up to translations.

By the maximality of $\mu_{q_n}$ it holds
$$F^{1/q_n}_{q_n}(B)\le F^{1/q_n}_{q_n}(\mu_{q_n})=T(\O_n)(\lambda(\O_n)+o(1))^{1/q_n}$$
and we deduce, passing to the limit as $n\to \infty$
$$
T(B)\le T(\mu)=\lim_{n\to\infty}T(\O_n).
$$
Since the sequence $q_n$ was arbitrary we obtain the conclusion.
\end{proof}

\begin{theo}\label{extDomain}
Let $\mu\in\M_{ad}$ be an optimal measure for $F_q$ with $q>1$. There exists $q_0>1$ such that for $q>q_0$ we have $\mu=I_{A_{\mu}}$. In particular the optimal measure can be represented by a quasi-open set.
\end{theo}

\begin{proof}

Since $F_q$ is scaling free, we can suppose that $|A_\mu|=1$. Let $\eps>0$ be a small parameter and let $\mu_\eps$ be the capacitary measure defined by
$$
\mu_\eps(E)=(1-\eps)\mu(E).
$$
Being $A_\mu=\A_{\mu_\eps}$ we have $\mu_\eps\in\M_{ad}$.
We assume by contradiction that $\mu\neq I_{A_\mu}$ (notice that this implies $\mu_\eps\neq \mu$). 
For the sake of brevity, we denote respectively by $w$ and $w_\eps$ the torsion functions of $\mu$ and $\mu_\eps$.
It is easy to verify that, as $\eps\to0$,
$$
\|\cdot \|_{H^1_{\mu_\eps}}\overset{\Gamma}{\to}\|\cdot \|_{H^1_{\mu}},\quad \hbox{ on } L^2(\R^d),
$$
and therefore we have $\mu_\eps\overset{\gamma}{\to} \mu$ and $w_\eps\to w$ in $L^{1}(\R^d)$, as $\eps\to 0$.
Let us denote by $t(\eps)$, $l(\eps)$ and $f_q(\eps)$ the real functions
$$
\eps\mapsto t(\eps)=T(\mu_\eps),\quad \eps\mapsto l(\epsilon)=\lambda(\mu_\eps),\quad \eps\mapsto f_q(\eps)=F_q(\mu_\eps),
$$ 
and by $t'_+(0)$, $l'_+(0)$, $(f_q)'_+(0)$ the limits for $\eps\to 0$ of the respective different quotients.

By writing $w_\eps=w+\eps\xi_\eps$ for some $\xi_\eps\in L^{1}(\R^d)$ and using the fact that $w, w_\eps$ respectively weakly solve the PDEs:
$$-\Delta w+w\mu=1,$$
\be\label{wdue}
-\Delta w_\eps+w_\eps\mu_\eps=1,
\ee
we deduce that $\xi_\eps$ weakly solves the PDE
\be\label{wtre}
-\Delta\xi_\eps+\xi_\eps\mu_\eps=w\mu.
\ee
This allows us to compute the derivative
\[\begin{split}
t_+'(0)&=\lim_{\eps\to 0}\left(\int \xi_\eps\,dx\right)=\lim_{\eps\to 0}\left(\int \nabla w_{\eps}\nabla \xi_\eps dx +\int w_{\eps}\xi_\eps d\mu_\eps\right)\\
&=\lim_{\eps\to 0}\left(\int w w_\eps d\mu\right),
\end{split}\]
where we test \eqref{wdue} with $\xi_\eps$ and we use \eqref{wtre} tested with $w_\eps$. Since, as $\eps\to 0$, $w_\eps\to w$ in $L^1(\R^d)$ we obtain
\be\label{tprimo}
t_+'(0)=\int w^2\,d\mu.
\ee
We can treat with a similar argument the eigenvalue. Let $u, u_\eps$ be the first eigenfunctions (with unitary $L^2$ norm) respectively of the operator $-\Delta+\mu_\eps$ and $-\Delta+\mu$ and let $v_\eps\in L^{2}(\R^d)$ be such that $u_\eps=u+\eps v_\eps$.
Since 
$$-\Delta u+u\mu=\la(\mu)u,\quad -\Delta u_\eps+u_\eps\mu_\eps=\la(\mu_\eps)u_\eps$$
we have
$$
-\Delta v_\eps +v_\eps\mu-u\mu-\eps v_\eps\mu=\left(\frac{\la(\mu_\eps)-\la(\mu)}{\eps}\right) u+\la(\mu_\eps)v_\eps.
$$
By testing the PDE above with $u\in H^{1}_\mu$ and since $\int u^2dx=1$, we obtain
$$
\left(\frac{\la(\mu_\eps)-\la(\mu)}{\eps}\right)=\int \nabla v_\eps\nabla u\,dx
+\int v_\eps u\,d\mu-\int u^2\,d\mu-\eps \int v_\eps u\,d\mu-\la(\mu_\eps)\int v_\eps u\,dx.$$
By taking the limit as $\eps\to 0$ and exploiting the fact that $u_\eps\to u$ weakly in $H^{1}_\mu$ and $\la(\mu_\eps)\to \la(\mu)$ we get
\be\label{lprimo}
l_+'(0)=-\int u^2\,d\mu.
\ee
By combining \eqref{tprimo} and \eqref{lprimo} we get
\[\begin{split}
&(f_q)_+'(0)=l'_+(0)T^q(\mu)+q\lambda(\mu)T^{q-1}(\mu)t'_+(0)=F_q(\mu)\int\left(-\frac{u^2}{\lambda(\mu)}+q\frac{w^2}{T(\mu)}\right)\,d\mu.
\end{split}\]
Now, the optimality condition on $\mu$ implies $(f_q)'(0)\le 0$ and hence that
\be\label{dacontraddire}
\int\left(\frac{u^2}{\lambda(\mu)}-q\frac{w^2}{T(\mu)}\right)\,d\mu\ge 0.
\ee
We claim that 
\be\label{claim}
\frac{u^2}{\lambda(\mu)}-q\frac{w^2}{T(\mu)}< 0 \quad\hbox{q.e on }\R^d
\ee
for $q$ large enough. Indeed, by an application of Theorem \ref{Davies} together with a comparison principle, we have
$$
u\leq e^{1/(8\pi)}\la^{d/4+1}(\mu)w \quad\hbox{q.e on }\R^d,
$$
and so by the P\'olya inequality
$$
u^2\le e^{1/(4\pi)}\la^{d/2}(\mu)\frac{\la(\mu)}{T(\mu)}w^2 \quad\hbox{q.e on }\R^d.
$$
The latter implies that
\[
\frac{u^2}{\lambda(\mu)}-q\frac{w^2}{T(\mu)}\leq \frac{w^2}{T(\mu)}\left(e^{1/(4\pi)}\la^{d/2}(\mu) -q \right) \quad\text{q.e. on }\R^d.
\]
Therefore, for every $q$ such that
$$\sup_{\mu\in\M_{ad}}e^{1/(4\pi)}\la^{d/2}(\mu)<q,$$
\eqref{claim} is verified. Notice that the supremum in the inequality above is finite as a consequence of Lemma \ref{qtoinfty} combined again with P\'olya inequality.

To conclude it is now enough to notice that \eqref{claim} contradicts \eqref{dacontraddire}.
\end{proof}

\section{Optimality for nearly spherical domains}\label{ssph}

In the following we consider the classes $\S_{\delta,\gamma}$ of {\it nearly spherical domains}. Let $B_1$ be the unitary ball of $\R^d$. A domain $\O$ such that
$$|\O|=|B_1|, \quad \int_\O xdx=0,$$
belongs to the class $\S_{\delta,\gamma}$ if there exists $\phi\in C^{2,\gamma}(\pa B_1)$ with $\|\phi\|_{L^{\infty}(\pa B_1)}\le 1/2$ and such that
$$\pa\O=\{x\in \R^d:\ x=(1+\phi(y))y,\ y\in\pa B_1\}, \quad \|\phi\|_{C^{2,\gamma}}(\pa B_1)\le \delta. $$
We recall the following result.

\begin{theo}\label{BraDePhiVel}
Let $\gamma\in (0,1)$. There exists $\delta=\delta(d,\gamma)>0$ such that if $\O\in \S_{\delta,\gamma}$ 
then 
\[
\begin{split}
&T(B_1)-T(\O)\ge C_1\|\phi\|^2_{H^{1/2}(\pa B_1)} \\
&\la(\O)-\la(B_1)\le C_2\|\phi\|^2_{H^{1/2}(\pa B_1) 
}
\end{split}
\]
for suitable constants $C_1$ and $C_2$ depending only on the dimension $d$.
\end{theo}



\begin{proof}
The inequality for the torsional rigidity follows from Theorem 3.3 in \cite{BDV} while the inequality for the eigenvalue follows by combining Theorem 1.2 and Lemma 2.8 of \cite{DL}.
\end{proof}

\begin{theo}\label{maxBall}
Let $\gamma\in (0,1)$. There exists $\delta>0$ and $q_1>1$ such that for every $q\ge q_1$ and every $\O\in \S_{\gamma,\delta}$ 
it holds
$$
\la(B_1)T^q(B_1)\ge \la(\O)T^{q}(\O).
$$
\end{theo}

\begin{proof}
For every domain $\O$ we have
$$
\la(B_1)T^{q}(B_1)-\la(\O)T^{q}(\O)=\la(B_1)(T^{q}(B_1)-T^{q}(\O))+T^{q}(\O)(\la(B_1)-\la(\O)),
$$
which, by the elementary inequality
$$
x^{q}-y^{q}\ge q y^{q-1}(y-x), \quad \hbox{for every }x,y\ge0,\ q>1,
$$
implies
\be\label{dallaconvessita}
\la(B_1)T^{q}(B_1)-\la(\O)T^{q}(\O)\ge T^{q-1}(\O)[q(T(B_1)-T(\O))-T(\O)(\la(\O)-\la(B_1))].
\ee

Let $\delta$ the constant determined by Theorem \ref{BraDePhiVel} and assume $\O\in \S_{\gamma,\delta}$. Since $2^{-1}B_1\subset\O\subset 2B_1$, we get $$2^{-(2+d)}T(B_1)\le T(\O)\le 2^{2+d}T(B_1).$$
Combining Theorem \ref{BraDePhiVel} and inequality \eqref{dallaconvessita} we get
$$\la(B_1)T^{q}(B_1)-\la(\O)T^{q}(\O)\ge (2^{-(2+d)}T(B_1))^{q-1}(qC_1- 2^{2+d}C_2T(B_1))\|\phi\|^2_{H^{1/2}(\pa B_1)}.$$
Hence, if $q$ is such that
$$q\ge2^{d+2}\frac{C_2}{C_1}T(B_1),$$
we obtain 
$$\la(B_1)T^{q}(B_1)\ge \la(\O)T^{q}(\O)$$
and this concludes the proof.
\end{proof}

\begin{rema}
Although for large $q$ we expect the ball to be optimal for the functional $F_q$, it is easy to see that this does not occur when $q$ approaches $1$. Indeed, if the ball maximizes $F_q$ for every $q>1$, passing to the limit as $q\to1$, this would happen also for $q=1$, which is not true, even in the class of convex domains. To see this it is enough to notice that
$$F_1(B_1)=\frac{\lambda(B_1)}{d(d+2)}\le\frac{d+4}{2(d+2)},$$
where the last inequality follows simply by taking $u(x)=1-|x|^2$ as a test function for $\lambda(B_1)$. On the other hand, taking as $\O_\eps$ the thin slab $]0,1[^{d-1}\times]0,\eps[$, gives
$$\lim_{\eps\to0}F_1(\O_\eps)=\frac{\pi^2}{12}$$
and
$$\frac{\pi^2}{12}>\frac{d+4}{2(d+2)}\qquad\text{for every }d\ge2.$$
\end{rema}

\bigskip

\noindent{\bf Acknowledgments.} SG gratefully acknowledges Professor Dorin Bucur for the stimulating discussions, and the hospitality, during a research period, at Universit\'e de Savoie-Mont Blanc. The work of GB is part of the project 2017TEXA3H {\it``Gradient flows, Optimal Transport and Metric Measure Structures''} funded by the Italian Ministry of Research and University. The authors are members of the Gruppo Nazionale per l'Analisi Matematica, la Probabilit\`a e le loro Applicazioni (GNAMPA) of the Istituto Nazionale di Alta Matematica (INdAM).

\bigskip
{\small\noindent
Luca Briani:\\
Dipartimento di Matematica, Universit\`a di Pisa\\
Largo B. Pontecorvo 5, 56127 Pisa - ITALY\\
{\tt luca.briani@phd.unipi.it}

\bigskip
\small\noindent
Giuseppe Buttazzo:\\
Dipartimento di Matematica, Universit\`a di Pisa\\
Largo B. Pontecorvo 5, 56127 Pisa - ITALY\\
{\tt giuseppe.buttazzo@dm.unipi.it}\\
{\tt http://www.dm.unipi.it/pages/buttazzo/}

\bigskip
\small\noindent
Serena Guarino Lo Bianco:\\
Dipartimento di Scienze F.I.M., Universit\`a degli Studi di Modena e Reggio Emilia\\
Via Campi 213/A, 41125 Modena - ITALY\\
{\tt serena.guarinolobianco@unimore.it}\\

\end{document}